\documentclass[11pt]{amsart}
\usepackage[utf8]{inputenc}
\usepackage[T1]{fontenc}
\usepackage[a4paper,margin=2.5cm]{geometry}
\usepackage{amsmath}
\usepackage{amssymb}
\usepackage{amsthm}
\usepackage{enumerate}
\usepackage{color}
\usepackage{graphicx}
\usepackage{hyperref}
\usepackage{diagbox}
\usepackage{pdflscape}
\usepackage{lipsum}

\usepackage{tikz}
\newcommand*\circled[1]{\tikz[baseline=(char.base)]{
            \node[shape=circle,draw,inner sep=2pt] (char) {#1};}}
            
\usepackage{caption}
\usepackage{geometry}
\usepackage{array}
\usepackage{tabularx}
\usepackage{booktabs}
\usepackage{makecell}

\newcolumntype{Y}{>{\centering\arraybackslash}X}

\renewcommand{\ge}{\geqslant}
\renewcommand{\geq}{\geqslant}
\renewcommand{\le}{\leqslant}
\renewcommand{\leq}{\leqslant}
\theoremstyle{plain}
\newtheorem{thm}{Theorem}[section]

\newtheorem{prop}[thm]{Proposition}
\newtheorem{cor}[thm]{Corollary}
\newtheorem{prob}[thm]{Problem}

\newtheorem{conj}[thm]{Conjecture}
\newtheorem{rem}[thm]{Remark}
\newtheorem{exmp}[thm]{Example}
\theoremstyle{definition}

\begin{document}
\author[J. Chappelon]{Jonathan Chappelon}
\address{IMAG, Univ.\ Montpellier, CNRS, Montpellier, France}
\email{\href{mailto:jonathan.chappelon@umontpellier.fr}{jonathan.chappelon@umontpellier.fr}}

\author[J.L. Ram\'irez Alfons\'in]{Jorge L. Ram\'irez Alfons\'in}
\address{IMAG, Univ.\ Montpellier, CNRS, Montpellier, France}
\email{\href{mailto:jorge.ramirez-alfonsin@umontpellier.fr}{jorge.ramirez-alfonsin@umontpellier.fr}}

\author[D.I. Stamate]{Dumitru I. Stamate}
\address{Faculty of Mathematics and Computer Science, University of Bucharest, Str. Academiei 14, Bucharest - 010014, Romania}
\email{\href{mailto:dumitru.stamate@fmi.unibuc.ro}{dumitru.stamate@fmi.unibuc.ro}}

\title{Numerical semigroup tree: type-representation}
\date{\today}
\subjclass[2010]{05A15, 05E40, 11D07}
\keywords{Numerical Semigroup, Tree}

\begin{abstract} In this paper, we introduce a new depicting of the so-called numerical semigroup tree $\mathcal T$. By exploring computationally this
improved picture, relying on the type notion of a semigroup, we found that the number of semigroups of genus $g$ and type $t$  is constant when $t$ is close to $g$ while $g$ grows. We also study the {unimodality} of various sequences as well as the behavior of the leaves in $\mathcal T$. We put forward several conjectures that are supported by various computational experiments.
\end{abstract}

\maketitle
\section{Introduction}%

A {\em numerical semigroup} $\Lambda$ is a subset of the non-negative integers $\mathbb{N}_0$, containing $0$, closed under addition and with finite complement in $\mathbb{N}_0$. 
{It is known that each numerical semigroup has a unique minimal system of generators. The cardinality of the minimal system of generators is called the {\em embedding dimension} and denoted $e(\Lambda)$.}
The elements of $\mathbb{N}_0$ that are not in $\Lambda$ are called the {\em gaps}. {We denote by $Gap(\Lambda)$ the set of gaps of $\Lambda$}. The number of gaps $g(\Lambda)$ is {called} the {\em genus} of $\Lambda$.
A relevant invariant of $\Lambda$ is  the so-called {\em Frobenius number} $F(\Lambda)$ defined as the largest gap in $\Lambda$ (we refer the reader to the book \cite{Ram} for a {thorough} discussion on the Frobenius number).
\smallskip

The {\em numerical semigroup tree}, first introduced in Rosales' thesis \cite{Ros1} (and discussed in \cite[Introduction]{Ros2}), is the infinite tree $\mathcal{T}$ whose root is the semigroup $\mathbb{N}_0$ and in
which the {\em parent} of a numerical semigroup $\Lambda$ is the numerical semigroup $\Lambda'$ obtained by adjoining to $\Lambda$ its
Frobenius number. In turn, the {\em descendants} of a numerical semigroup $\Lambda$ are the numerical semigroups obtained by taking away one by one the {minimal} generators that are larger than $F(\Lambda)$.  For instance, the descendants of $\Lambda=\langle 3,7,8\rangle$ are $\Lambda'=\langle 3,7,11\rangle$ and $\Lambda''=\langle 3,8,10\rangle$ (or equivalently, the parent of $\Lambda'$ and $\Lambda''$ is $\Lambda$). Indeed, since $7,8>F(\Lambda)=5$ are {minimal} generators of $\Lambda$ then $\Lambda'=\Lambda\setminus \{8\}$ and $\Lambda''=\Lambda\setminus \{7\}$ (or equivalently,
 $\Lambda=\Lambda'\cup {\{} F(\Lambda'){\}}=\Lambda''\cup {\{}F(\Lambda''){\}}$ with $F(\Lambda') =8$ and $F(\Lambda'') =7$).
We notice that all numerical semigroups are in $\mathcal{T}$ and the parent of a numerical semigroup of genus $g$ has genus $g-1$, see Figure \ref{fig1}. 

 \begin{figure}[ht] 
\centering
 \includegraphics[width=1\textwidth]{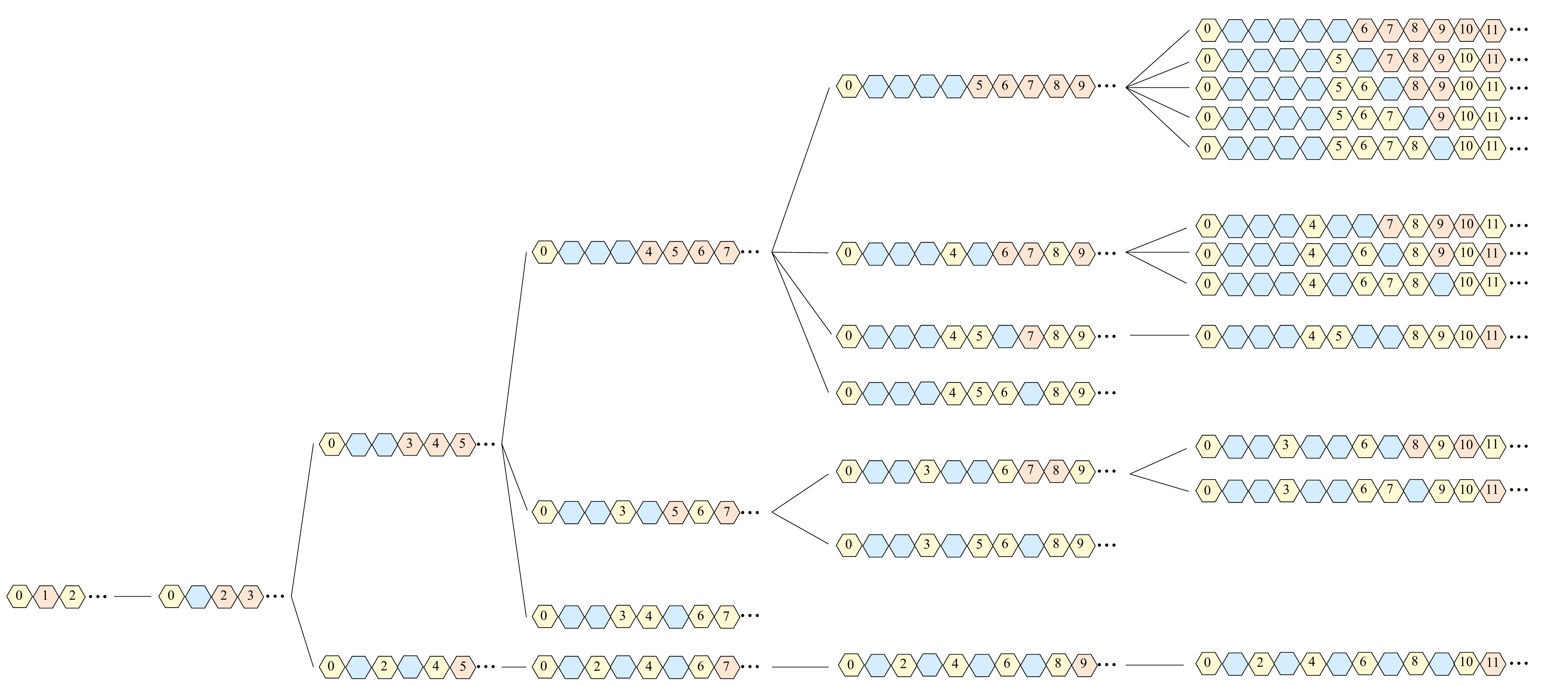}
\caption{First levels of tree $\mathcal T$ where yellow, blue and red hexagons denote elements, gaps and generators greater than the Frobenius number of the corresponding semigroup respectively.}\label{fig1}
\end{figure} 

The tree $\mathcal{T}$ has been the object of many investigations in different contexts: {Frobenius pseudo-varieties} \cite{RR}, number of numerical semigroups of given genus 
 \cite {BA,EF,FH}, Wilf's conjecture \cite{DEF, Ka1,Ka2}, etc.
\smallskip

In this paper, we introduce a new approach to study $\mathcal T$. We put forward a new way to represent $\mathcal T$ by using the notion of {\em type} of a numerical semigroup. This representation {reveals} unexpected behaviors of numerical semigroups according with both their type and genus. We hope (and expect) that this investigation leads not only to a better comprehension of $\mathcal T$ but also to open new methods to attack well-known open problems. 

The type of a numerical semigroup $\Lambda$ is defined in \cite{FGH} as the cardinality of the set of its pseudo-Frobenius numbers. The former also equals the Cohen-Macaulay type of the semigroup ring $K[\Lambda]$, where $K$ is any field. In general, it is difficult to estimate the type in terms of the minimal generators of the semigroup, see \cite{S-betti}.
 
In the next section, we introduce the type-representation depicting $\mathcal T$ and present some preliminary results on the parent-descendants types relationship. 

Let $n(g,t)$ be the number of numerical semigroups of genus $g$ and type $t$. In Section \ref{sec:stab}, we investigate a mysterious stabilizer behavior of the value $n(g,t)$ when $t$ is close to $g$ while $g$ grows, as indicated by the data in Table \ref{table1}.
 We show  that the value $n(g,t)$ {only depends on $g-t$} when $g$ is large enough and   $t\geq (2g-1)/3$ (Theorem \ref{thm:main1}). We also study the shape of the generators of these {\em stabilizers} semigroups (Theorem \ref{thm:main}, Propositions \ref{propo:1} and \ref{propo:2}) allowing us to obtain the lower bound  $n(g,g-\ell)\geq \ell^2-3\ell +10$ for $\ell\ge 6$ and $g\ge 3\ell-1$ (Theorem \ref{thm:main2}). 
One tool is to represent the gaps relevant for deciding the type of the semigroup in terms of what we call {\em gap vectors}. These are $(0/1)$-vectors of length $2(g-t)$ with equal number of $0$'s and $1$'s.  When $t\geq (2g-1)/3$ we show that the semigroup is essentially determined by its genus and the gap vector (Proposition \ref{prop:type}).

In Section \ref{sec:uni-leave}, we focus our attention to the study of the  {unimodality} property of various sequences, for instance, the one obtained by the number of semigroups with fixed genus and increasing type that are leaves in $\mathcal T$. We put forward a number of conjectures/problems supported by various experimental computations {using} the {\tt numericalsgps package} \cite{Num-semigroup} in {\tt GAP} \cite{GAP}.

\medskip
{\bf Acknowledgement}.
The third author was partially supported by the CNRS International Research Network ECO-Math.  He  thanks the University of Montpellier for the hospitality during the visits when this paper developed. {We would like to thank the referees for several helpful remarks.}
\medskip

\section{Type-representation}\label{sec:rep}

\subsection{Some preliminaries} 
The {\em type} $t(\Lambda)$ of $\Lambda$ is the cardinality of the set 
$$PF(\Lambda)=\{x \in\mathbb{Z}\setminus \Lambda \ | \ x + \ell \in \Lambda \text{ for all } \ell\in \Lambda\setminus\{0\}\}.$$ 

The elements of $PF(\Lambda)$ are called {\em pseudo-Frobenius} numbers. This concept was introduced in \cite{FGH}. Notice that $F(\Lambda)=\max\{p\ | \ p\in PF(\Lambda)\}$.
\smallskip
The partial order $\le_\Lambda$ induced by  $\Lambda$ on the integers is defined as follows: $x\le_\Lambda y$ if $y-x\in \Lambda$. It is well-known \cite{FGH} that 

1) $PF(\Lambda)= \text{Maximals}_{\le_\Lambda} (\mathbb{Z}\setminus \Lambda)$
 
2) $x\in\mathbb{Z}\setminus \Lambda$ if and only if $f-x\in \Lambda$ for some $f\in PF(\Lambda)$.

The {\em multiplicity} of $\Lambda$, denoted by $m(\Lambda)$, is the least positive element in $\Lambda$. 

If $\Lambda$ is a numerical semigroup with positive genus, that is, $\Lambda\subsetneq \mathbb{N}$, then $\Lambda\cup \{F(\Lambda)\}$ is also a numerical semigroup whose embedding dimension is {at most} one more that of $\Lambda$. Also, 
$$
m(\Lambda\cup\{F(\Lambda)\}) = 
\begin{cases} m(\Lambda)-1  &\text{ if }t(\Lambda)=g(\Lambda), \\   m(\Lambda) &\text{ otherwise}.
\end{cases}
$$
We have, by construction, that $g(\Lambda)>g(\Lambda\cup F(\Lambda))=g(\Lambda)-1$.

\subsection{New depicting of $\mathcal T$} 
Let $\Lambda$  be a numerical semigroup {in which} $x$ and $y$ are two minimal generators with $y,x>F(\Lambda)$.
{Often}, when depicting the tree $\mathcal T$, if $\Lambda'=\Lambda\setminus \{x\}, \Lambda''=\Lambda\setminus \{y\}$ are descendants of $\Lambda$ 
then $\Lambda'$ comes before $\Lambda''$ if $x<y$. 
\smallskip

We propose a different depicting of $\mathcal T$ according with the types.
The {\em type-representation} of $\mathcal T$ is given by ordering each level of $\mathcal T$ by increasing type order; see Figure \ref{fig2}. 

 \begin{figure}[ht] 
\centering
\includegraphics[width=1\textwidth]{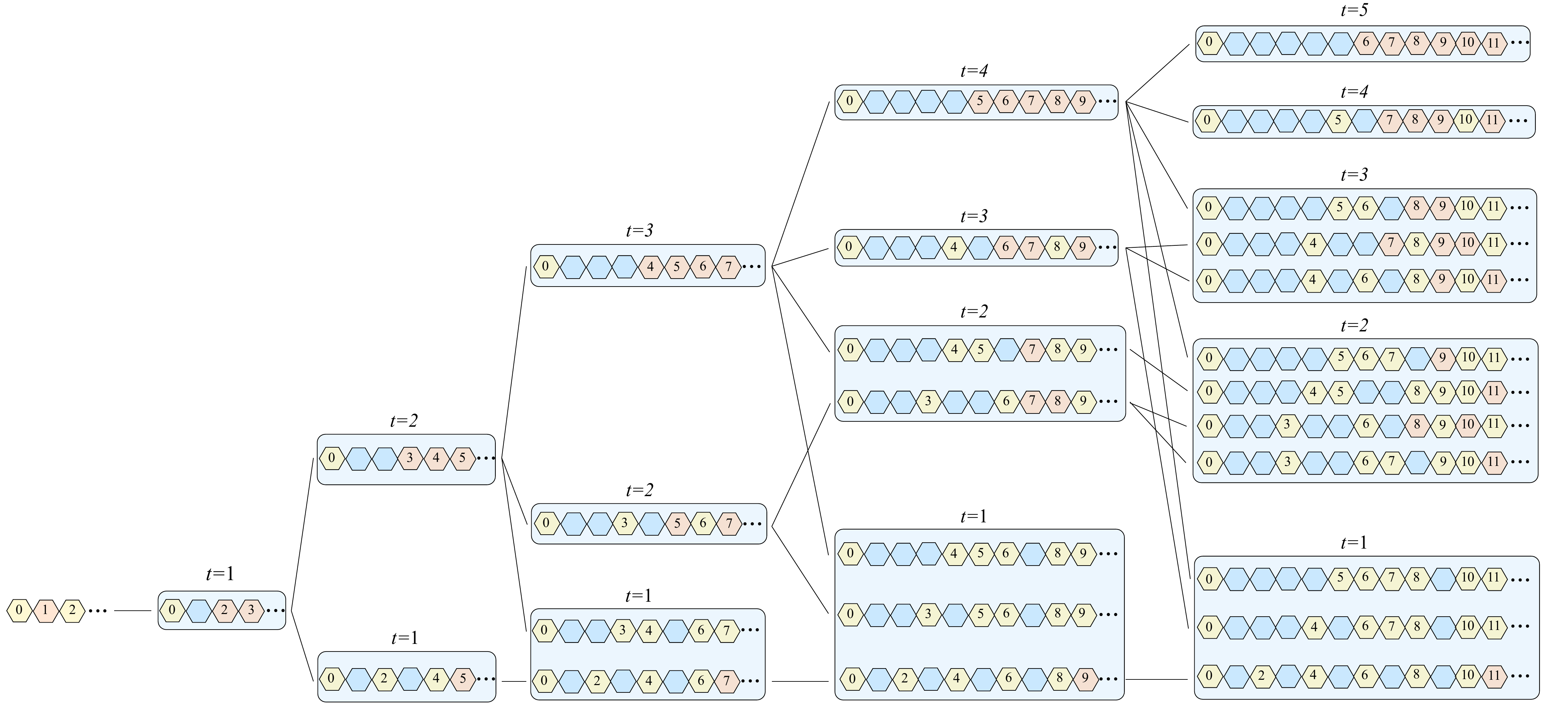}
\caption{First levels of the type-representation of $\mathcal T$. In each level, semigroups of the same type $t$ are grouped together in a box.}\label{fig2}
\end{figure}



\subsection{Parent-descendants types} We present a few results relating $t(\Lambda)$ and $t(\Lambda')$ where $\Lambda'$ is a descendant of $\Lambda$. We have the following

\begin{prop}\label{prop11} Let $\Lambda'$ be a numerical semigroup of genus $g>0$. Let $\Lambda=\Lambda'\cup\{F(\Lambda')\}$. Then,
\begin{enumerate}
\item[(i)] $PF(\Lambda')\setminus \{F(\Lambda')\} \subseteq PF(\Lambda)$,  with equality if and only if $t(\Lambda')=g(\Lambda')$.
\item[(ii)] $t(\Lambda)\geq t(\Lambda')-1$, with equality if and only if $t(\Lambda')=g(\Lambda')$.
\item[(iii)] If $m(\Lambda')<F(\Lambda')$ then $t(\Lambda) \geq t(\Lambda')$.
\end{enumerate}
\end{prop}

We first notice that, by definition, $PF(\Lambda) \subseteq Gaps(\Lambda)$ and thus obtaining the obvious bound $t(\Lambda)\leq g(\Lambda)$. The equality can be characterized as follows.

\begin{prop}\label{gist} Let $\Lambda$ be a numerical semigroup. Then, the following statements are equivalent.
\begin{enumerate}
\item[(i)] $t(\Lambda)=g(\Lambda)$.
\item[(ii)] $m(\Lambda)=F(\Lambda)+1$.
\item[(iii)] $F(\Lambda)<m(\Lambda)$.
\item[(iv)] $\Lambda=\langle c, c+1,\dots, 2c-1\rangle$ for some integer $c\geq 1$. 
\end{enumerate}
If these hold then $g(\Lambda)=t(\Lambda)=F(\Lambda)=c-1$. 
\end{prop}

\begin{proof} The equality $t(\Lambda)=g(\Lambda)$ means that all gaps in $\Lambda$ are  pseudo-Frobenius numbers, that is, these are not comparable with respect to the $\leq_\Lambda$ partial order. In particular, for any gap $x<m(\Lambda)$ we get that $x+m(\Lambda)$ is not a gap. Thus $(i)$ is equivalent to the fact that $1,\dots, m(\Lambda)-1$ are all the gaps in  $\Lambda$, which implies $(ii)$. Clearly, $(ii) \implies (iii)\implies  (iv) \implies (i)$, and from the latter we get that $c-1=t(\Lambda)=g(\Lambda)=F(\Lambda)$.
\end{proof}

{\em Proof of Proposition \ref{prop11}}.
For $(i)$ we distinguish two cases. First, assume $t(\Lambda')=g(\Lambda')$. By Proposition \ref{gist} we get $\Lambda'=\langle c, c+1,\dots, 2c-1\rangle$ with $c=g(\Lambda')+1=F(\Lambda')+1$, and $\Lambda=\langle c-1, c,\dots, 2c-3\rangle$. Hence, $PF(\Lambda')\setminus\{F(\Lambda')\}=\{1,\dots, c-2\}= PF(\Lambda)$.

Now, assume $t(\Lambda')<g(\Lambda')$, that is, $m(\Lambda')<F(\Lambda')$ by Proposition \ref{gist}.   We verify that $F(\Lambda')-m(\Lambda')\in PF(\Lambda)\setminus PF(\Lambda')$. Indeed, $F(\Lambda')-m(\Lambda')<_{\Lambda'} F(\Lambda') \in PF(\Lambda')$, hence $F(\Lambda')-m(\Lambda')\notin PF(\Lambda')$. Also, for any $0\neq h\in \Lambda$ we have $F(\Lambda')-m(\Lambda')+h\geq F(\Lambda')$, thus $F(\Lambda')-m(\Lambda')\in PF(\Lambda)$. 

Let $a\in PF(\Lambda')\setminus \{F(\Lambda')\}$. Therefore, $a\notin \Lambda'\cup \{F(\Lambda')\}$ and $a+(\Lambda'\setminus\{0\}) \subset \Lambda'$. As $a+F(\Lambda')\geq F(\Lambda')+1$, we obtain $a+F(\Lambda')\in \Lambda$. This implies that $a+ (\Lambda\setminus\{0\})\subset \Lambda$, that is, $a\in PF(\Lambda)$.
Consequently, $PF(\Lambda')\setminus\{F(\Lambda')\} \subsetneq PF(\Lambda)$. This finishes the proof of part $(i)$.

Part  $(ii)$ follows at once from $(i)$, and also part $(iii)$ by using Proposition \ref{gist}.
\hfill$\square$

\section{Stabilizer sequence}\label{sec:stab}

Let $n(g,t)$ be the number of numerical semigroups of genus $g$ and type $t$. In Table \ref{table1} are computed with GAP (\cite{GAP}, \cite{Num-semigroup}) the values of $n(g,t)$ for each $1\le t\le g\le 33$.
We notice that the right diagonals (bold numbers) in Table \ref{table1} become constant at certain threshold. These values are given in Table \ref{table:2}.

\begin{table}[htbp!]
\centering
\begin{tabular}{c|c|c|c|c|c|c|c|c|c|c|c}
$g\ge$  & 2 &  5 &  8 &  11 &  14 &  17 &20 & 23&  26 & 29 &  32 \\  
\hline
$t$ & $g-1$ & $g-2$ &$g-3$ &$g-4$ &$g-5$ &$g-6$ &$g-7$ &$g-8$ &$g-9$ &$g-10$ &$g-11$\\
\hline
$n(g,t)$ &1 & 3 & 7 & 15 & 35 & 78 & 161 & 367 & 757 & 1632 & 3436\\
\end{tabular}
\caption{Some values of $n(g,t)$.}
\label{table:2}
\end{table} 


The values in Table \ref{table:2} suggest that for $g$ large enough and $t$ close to $g$ then $n(g,t)$ remains constant.
Before establishing the latter formally  in Theorem \ref{thm:stabilizer} below, we discuss some connections between the type and genus of a numerical semigroup.
We first show the following inequality. 
{This is known \cite[Proposition 2.2]{Nari}; also see \cite{DGR}. Nevertheless, we include a proof for completeness.}

\begin{prop}\label{fgt-bound}
For any numerical semigroup $\Lambda$ we have
$$
F(\Lambda)+t(\Lambda)\leq 2g(\Lambda).
$$
\end{prop}

\begin{proof}
Set $g=g(\Lambda)$ and $t=t(\Lambda)$. We denote $A=\Lambda\cap [1,\dots, F(\Lambda)]$ and $B=Gaps(\Lambda)\setminus PF(\Lambda)$.  Clearly,  $|B|=g-t$, and since $PF(\Lambda)\cup A\cup B$ is a partition of $[1,\dots, F(\Lambda)]$ it follows that $|A|=F(\Lambda)-g$.
The map
$$
\varphi :(PF(\Lambda)\setminus\{F(\Lambda)\})\cup A \to Gaps(\Lambda)\setminus\{F(\Lambda)\}
$$
given by $\varphi(x)=F(\Lambda)-x$ is well defined and injective, hence 
$$|(PF(\Lambda)\setminus\{F(\Lambda)\})\cup A | \leq |Gaps(\Lambda)\setminus\{F(\Lambda)\}|,
$$ 
that is $(t-1)+(F(\Lambda)-g)\leq g-1$, leading to the desired inequality.
\end{proof}



Let $\Lambda$ be a numerical semigroup with genus $g(\Lambda)>0$ and type $t(\Lambda)=g(\Lambda)-\ell$, for $0\leq \ell\leq g-1$. 
The case $\ell=0$ is fully understood by Proposition \ref{gist}. 

Assume now $\ell>0$. 
Since $m(\Lambda)\geq t(\Lambda)+1$ (cf. \cite[Corollary 2.23]{RG}) and, by Proposition \ref{fgt-bound},  we have 
$$F(\Lambda)\leq 2g(\Lambda)-t(\Lambda)=g(\Lambda)+\ell,$$ we obtain  
$$
\{1,\dots, g-\ell\} \subseteq Gaps(S) \subseteq \{1,\dots, g+\ell\}.
$$ 

Motivated by these inclusions, we may attach to $\Lambda$ the $(0/1)$-vector  $v(\Lambda)=(v_1,\dots, v_{2\ell})$, where
$$v_i=
\begin{cases} 1 \text{ if } g-\ell+i \in Gaps(\Lambda), \\
0 \text{ otherwise}.
\end{cases}
$$
It is clear that $\ell$ of the entries in $v(\Lambda)$ are equal $1$, and the rest of $\ell$ entries are $0$.
\smallskip

We call $v(\Lambda)$ the {\em  gap vector} of $\Lambda$ {(a similar binary encoding was used in \cite{AF})}.

For any $(0/1)$-vector $v=(v_1,\dots, v_{2\ell})$ and any integer $g> \ell$  we denote
\begin{eqnarray*}
\mathcal{G}_{g,v} &=& \{1,\dots, g-\ell\}\cup\{g-\ell+i: 1\leq i \leq 2\ell \text{  and } v_i=1\}, \\
\Lambda_{g,v} &=& \mathbb{N}_0\setminus \mathcal{G}_{g,v}.
\end{eqnarray*}

The semigroup $\Lambda$ and its gaps  
can be recovered from $v(\Lambda)$ and $g(\Lambda)$ since 
$$Gaps(\Lambda)= \mathcal{G}_{g(\Lambda),v(\Lambda)} \text{  and  } \Lambda=\Lambda_{g(\Lambda), v(\Lambda)}.$$

\begin{rem} {\em 
(a) It is not always true that  the set $\mathcal{G}_{g, v}$ is the set of gaps of a numerical semigroup. Indeed, for $v=(0,1,\dots, 0, 1)\in \mathbb{Z}^{2\ell}$ and $g=3\ell-2$ the set
$$
\mathcal{G}_{3\ell-2, v}=\{1,\dots, g-\ell=2\ell-2\}\cup\{2\ell, 2\ell+2, 2 \ell+4, \dots, g+\ell=4\ell-2\}
$$
is not the set of gaps of a numerical semigroup since $g-\ell+1=2\ell-1 \notin \mathcal{G}_{3\ell-2, v}$, yet $2(2\ell-1)=4\ell -2 \in \mathcal{G}_{3\ell-2, v}$.
\smallskip

b) {Even if}  $\mathcal{G}_{g,v}$ {were} the set of gaps of a numerical semigroup, its type {could} be different from $\ell$. Indeed, for the vector $v=(1,\dots, 1, 0,\dots,0)$ with $\ell$ $1$'s and $\ell$ $0$'s, and any $g>\ell>1$ one has 
$$
\mathcal{G}_{g, v}=\{1,2,\dots, g\},
$$
and its complement $\Lambda_{g, v}=\{0, g+1, g+2, \dots\}$ is a numerical semigroup of type equal to $g$, by Proposition \ref{prop11}. }
\end{rem}


\begin{prop} \label{prop:type} Let $\ell\ge 1$ and $g\geq 3\ell -1$ be positive integers and let $v=(v_1,\dots, v_{2\ell})$ be any vector where half of the entries are $0$ and the other half are $1$. Then,
\begin{enumerate}
\item[(i)] $\Lambda_{g,v}$ is a numerical semigroup of genus $g$.
\item[(ii)] $Gaps(\Lambda_{g,v}) \setminus PF(\Lambda_{g,v})=\{  j-i: 1\leq i<j\leq 2\ell, v_i=0, v_j=1\}$.
\item[(iii)]  $t(\Lambda_{g,v})=g-|\{  j-i: 1\leq i<j\leq 2\ell, v_i=0, v_j=1\}|$.
\end{enumerate}
\end{prop}

\begin{proof} We write $\Lambda=\Lambda_{g,v}$, for short.  

$(i)$ We clearly have that $2\min(\Lambda\setminus \{0\}) \geq 2(g-\ell+1)$ and $\max \mathcal{G}_{g,v}\leq g+\ell$.
The hypothesis $g\geq 3\ell-1$ gives $2(g-\ell+1)\geq g+\ell+1$, and therefore $\Lambda$ is a numerical semigroup with $Gaps(\Lambda)=\mathcal{G}_{g,v}$ and genus $g(\Lambda)=g$.

$(ii)$ Let $a\in Gaps(\Lambda)\setminus PF(\Lambda)$. Then there exists $b\in Gaps(\Lambda)$ and $0\neq \lambda \in \Lambda$ with $a+\lambda=b$. Since $g+\ell \geq b > \lambda  \geq m(\Lambda)\geq g-\ell+1$,  
we may write $b=g-\ell+j$, $ \lambda=g-\ell+i$ with $1\leq i<j\leq 2\ell$, $v_j=1$ and $v_i=0$. Therefore, $a=b-\lambda=j-i$.

Conversely, if $a=j-i$ with $1\leq i <j\leq 2\ell$ such that $v_j=1$ and $v_i=0$, we may write $a=(g+\ell+j)-(g+\ell+i)$ with $g+\ell+j\in Gaps(\Lambda)$ and $g+\ell+i\in \Lambda$.
\smallskip

We clearly have that $(iii)$ is a straight consequence of $(ii)$.
\end{proof}

The previous result motivates the following definition. For any $(0/1)$-vector $v=(v_1,\dots, v_{2\ell})$, its {\em cotype} is the integer
$$
cotype(v)=|\{ j-i:1\leq i<j\leq 2\ell \text{ with } v_i=0 \text{ and } v_j=1\}|.
$$

Under the assumptions of Proposition \ref{prop:type}, we have
\begin{equation} \label{eq:t}
type(\Lambda_{g,v})= g-cotype(v).
\end{equation}

\begin{exmp}  \label{ex:cotype} The following vectors have $\ell \geq 1$ entries equal to $1$ and $\ell$ entries equal to $0$.
\begin{enumerate}
\item[(i)] $cotype(1,\dots, 1, 0,\dots, 0) =|\emptyset|=0$.
\item[(ii)] $cotype(0,\dots, 0, 1, \dots, 1)=|\{1,\dots, \ell, \ell+1,\dots, 2\ell -1\}|=2\ell-1$.
\item[(iii)] $cotype(0,1,0,1, \dots, 0, 1)=|\{1,3,\dots, 2\ell-1\}|=\ell$.
\end{enumerate}
\end{exmp}

\begin{thm}\label{thm:main} Let $g\geq t>0$ be integers with $t\geq (2g-1)/3$ (or equivalently $g\leq (3t+1)/2$).
\begin{enumerate}
\item[(a)] If  $\Lambda$ is any numerical semigroup of genus $g$ and type $t$, then
$$
\Gamma=\{0\}\cup \{x+1: 0\neq x\in \Lambda\} 
$$
is a numerical semigroup of genus $g+1$ and type $t+1$.
\item[(b)] If $\Gamma$ is any numerical semigroup of genus $g+1$ and type $t+1$, then
$$
\Lambda=\{0\} \cup \{x-1: 0\neq x \in \Gamma\}
$$
is a numerical semigroup of genus $g$ and type $t$.
\end{enumerate}
\end{thm}

We give two proofs of Theorem \ref{thm:main}. The first one is quite {straightforward}. However, it does not give much information about the shape of the generators of these stabilizer semigroups. The second one allows us to {better} understand such generators. 

{\em First proof of Theorem  \ref{thm:main}.}
First, suppose that $\Lambda$ is a numerical semigroup of genus $g$ and type $t$. Let $\gamma_1$ and $\gamma_2$ be two positive elements of $\Gamma$. Since $\gamma_1-1$ and $\gamma_2-1$ are in $\Lambda$ and since $m(\Lambda)\ge t+1$ (cf. \cite[Corollary 2.23]{RG}), we have that
$$
\gamma_1+\gamma_2-1 = (\gamma_1-1)+(\gamma_2-1) + 1 \ge 2m(\Lambda) +1 \ge 2t+3.
$$
Moreover, since $3t+1\ge 2g$ by hypothesis and since $F(\Lambda)+t\le 2g$ by Proposition~\ref{fgt-bound}, we obtain that
$$
\gamma_1+\gamma_2-1 \ge 2t+3 \ge 2g-t+2 \ge F(\Lambda)+2.
$$
It follows that $\gamma_1+\gamma_2-1\in\Lambda$ and thus $\gamma_1+\gamma_2\in\Gamma$. Therefore $\Gamma$ is closed under addition. Since $|\mathbb{N}_0\setminus\Gamma|=g+1$ by definition, we have that $\Gamma$ is a numerical semigroup of genus $g+1$.

Conversely, suppose that $\Gamma$ is a numerical semigroup of genus $g+1$ and type $t+1$. Let $\lambda_1$ and $\lambda_2$ be two positive elements of $\Lambda$. Since $\lambda_1+1$ and $\lambda_2+1$ are in $\Gamma$ and since $m(\Gamma)\ge t+2$ (cf. \cite[Corollary 2.23]{RG}), we have that
$$
\lambda_1+\lambda_2+1 = (\lambda_1+1)+(\lambda_2+1) - 1 \ge 2m(\Gamma) -1 \ge 2t+3.
$$
Moreover, since $3t+1\ge 2g$ by hypothesis and since $F(\Gamma)+t+1\le 2g+2$ by Proposition~\ref{fgt-bound}, we obtain that
$$
\lambda_1+\lambda_2+1 \ge 2t+3 \ge 2g-t+2 \ge F(\Gamma)+1.
$$
It follows that $\lambda_1+\lambda_2+1\in\Gamma$ and thus $\lambda_1+\lambda_2\in\Lambda$. Therefore $\Lambda$ is closed under addition. Since $|\mathbb{N}_0\setminus\Lambda|=g$ by definition, we have that $\Lambda$ is a numerical semigroup of genus $g$.

For the types, we can show that
\begin{equation}\label{eq0}
Gaps(\Lambda)\setminus PF(\Lambda) = Gaps(\Gamma)\setminus PF(\Gamma).
\end{equation}
Let $a\in Gaps(\Lambda)\setminus PF(\Lambda)$. Since $a\notin PF(\Lambda)$, there exists $b\in\Lambda^*$ such that $a+b\in Gaps(\Lambda)$. Since $a+b+1\in Gaps(\Gamma)$ with $b+1\in\Gamma^*$, it follows that $a\in Gaps(\Gamma)\setminus PF(\Gamma)$.

Conversely, let $a\in Gaps(\Gamma)\setminus PF(\Gamma)$. Since $a\notin PF(\Gamma)$, there exists $b\in\Gamma^*$ such that $a+b\in Gaps(\Gamma)$. If $b=1$, then $\Gamma=\mathbb{N}_0$ of genus $0$, in contradiction with the hypothesis that $\Gamma$ is of genus $g+1$. Since $a+b-1\in Gaps(\Lambda)$ with $b-1\in\Lambda^*$, it follows that $a\in Gaps(\Lambda)\setminus PF(\Lambda)$.

From \eqref{eq0}, we obtain that
$$
g-t(\Lambda)=|Gaps(\Lambda)\setminus PF(\Lambda)|=|Gaps(\Gamma)\setminus PF(\Gamma)| = g+1-t(\Gamma).
$$
Therefore, we have $t(\Lambda)=t(\Gamma)-1$. This completes the proof.
\hfill$\square$
\smallskip

{\em Second proof of Theorem  \ref{thm:main}.} 
Set $\ell=g-t$. The condition $t\geq (2g-1)/3$ is equivalent to $g\geq 3\ell -1$.

(a):  If $\ell=0$, then by Proposition \ref{gist} we have   $\Lambda=\{0, g+1, g+2, \dots \}$. Hence $\Gamma=\{0, g+2, g+3,\dots\}$ has genus and type equal to $g+1$, again by Proposition \ref{gist}.

Assume $\ell >0$. Let $v=(v_1,\dots, v_{2 \ell})$ the gap vector of $\Gamma$. So, $\Lambda=\Lambda_{g, v}$ and $Gaps(\Lambda)=\mathcal{G}_{g,v}$. By construction, the set $\mathbb{N}_{0}\setminus \Gamma = \{1\} \cup \{x+1: x\in Gaps(\Lambda)\}$ has cardinality $g+1$ and the former equals $\mathcal{G}_{g+1, v}$. Therefore, $\Gamma=\Lambda_{g+1, v}$.  
As  $g+1 > g \geq 3 \ell -1$, by Proposition \ref{prop:type} we obtain that $\Gamma$ is a numerical semigroup of genus $g+1$ and using \eqref{eq:t} twice we get 
$$t(\Gamma)=g+1-cotype(v)= (g-cotype(v)) +1= t(\Lambda_{g,v})+1 = t(\Lambda)+1=t+1.$$
 
(b): Let $w$ be the gap vector of $\Gamma$. Then $\Gamma=\Lambda_{g+1, w}$ and $\Lambda=\Lambda_{g, w}$.
As $\ell=g-t=(g+1)-(t+1)$, and $g+1>g>3\ell-1$, from Proposition \ref{prop:type} applied twice we have that $\Lambda$ is a numerical semigroup of genus $g$ and $t(\Lambda)=g-cotype(w)=g- (g+1-t(\Gamma))=t(\Gamma)-1=t$.
\hfill$\square$
\smallskip

For integers $g,t\ge 1$, we denote by $\mathcal{L}(g,t)$ the set of numerical semigroups of genus $g$ and type $t$.
We notice that the second proof of Theorem \ref{thm:main} relies on the fact that the numerical semigroups $\Gamma$ and $\Lambda$ have the same gap vector. We have the following 

\begin{cor} \label{cor:largeg} Let $\ell $ and $g\geq 3\ell-1$ be integers. Then, the set 
$$\left\{v\in \mathbb{Z}^{2\ell}: v=v(\Lambda) \text{ for some }\Lambda\in \mathcal{L}(g, g-\ell) \right\}$$
does not depend on $g$.
\end{cor}

The vector $v\in \mathbb{Z}^{2 \ell}$ is called a {\em stable gap vector} if there exists $\Lambda$ a numerical semigroup of genus $g$ and type $g-\ell$ with $v=v(\Lambda)$ for some (and hence for all) $g\geq 3 \ell -1$. We denote  $\mathcal{V}(\ell)$ the set of stable gap vectors in $\mathbb{Z}^{2\ell}$.


\begin{thm}\label{thm:main1}
\label{thm:stabilizer}  Let $\ell $ and $g\geq 3\ell-1$ be integers. Then, $$n(g, g-\ell)=|\mathcal{V}(\ell)|$$
where  $\mathcal{V}(\ell)$ denotes the set of stable gap vectors in $\mathbb{Z}^{2\ell}$.
\end{thm}

\begin{proof}
Consider the map $\varphi:\mathcal{L}(g, g-\ell)\to \mathcal{V}(\ell)$ where $\varphi(\Lambda)=v(\Lambda)$. By Corollary \ref{cor:largeg},  
the image of $\varphi$ does not depend on $g$, thus $\varphi$ is surjective.
We note that if $\varphi(\Lambda)=v$ then $\Lambda=\Lambda_{g,v}$. This proves that $\varphi$ is injective.
\end{proof}
 
It would be interesting to determine in terms of $\ell$ the stabilized value of $n(g, g-\ell)$  and the set $\mathcal{L}(g, g-\ell)$ for $g\gg 0$.
According to Theorem \ref{thm:stabilizer}, this process goes hand in hand with finding the set $\mathcal{V(\ell)}$ of stable gap vectors for any $\ell$.
  The following two propositions are on this direction. 

\begin{prop}\label{propo:1} \label{q-small-ell} Let $\Lambda$ be a numerical semigroup of genus $g>0$. 
\begin{enumerate}
\item[(a)] $t(\Lambda)= g$ if and only if $\Lambda=\langle g+1,\dots, 2g+1\rangle$, that is, 
$$Gaps(\Lambda)=\{ 1,\dots, g\}.$$
\item[(b)] If $g\geq 2$ then, $t(\Lambda)=g-1$ if and only if $\Lambda=\langle g, g+2, g+3, \dots, 2g-1, 2g+1\rangle$, that is, 
\begin{eqnarray*}
v(\Lambda)=(0,1)  \text{ and  }  Gaps(\Lambda)&=&\{1,\dots, g-1, \star, g+1\}.
\end{eqnarray*}
\item[(c)] If $g\geq 5$ then, $t(\Lambda)= g-2$ if and only if
\begin{eqnarray*}
v(\Lambda)=(1,0,0,1) \text{ and  }  Gaps(\Lambda)&=& \{1,\dots, g-1, \star,\star, g+2\}, \text{ or} \\
v(\Lambda)=(0,1,1,0) \text{ and  }  Gaps(\Lambda)&=& \{ 1,2, \dots, g-2, \star, g, g+1 \}, \text{ or} \\
v(\Lambda)=(0,1,0,1) \text{ and  }  Gaps(\Lambda)&=& \{1,2, \dots, g-2, \star, g, \star, g+2\}.
\end{eqnarray*}
\item[(d)] If $g\geq 8$ then, $t(\Lambda)=g-3$ if and only if
\begin{eqnarray*}
v(\Lambda)=(1,1,0,0,0,1) \text{ and  } Gaps(\Lambda)&=&\{1,\dots, g-1, \star,\star,\star, g+3\}, \text{ or} \\
v(\Lambda)=(1,0,1,0,0,1) \text{ and  } Gaps(\Lambda)&=&\{1,\dots, g-2, \star,g,\star, \star,g+3\}, \text{ or} \\
v(\Lambda)=(1,0,0,1,1,0) \text{ and  } Gaps(\Lambda)&=&\{1,\dots, g-2, \star,\star, g+1, g+2\}, \text{ or} \\
v(\Lambda)=(0,1,1,1,0,0) \text{ and  } Gaps(\Lambda)&=&\{1,\dots, g-3, \star,g-1, g, g+1\}, \text{ or} \\
v(\Lambda)=(0,1,1,0,1,0) \text{ and  } Gaps(\Lambda)&=&\{1,\dots, g-3, \star, g-1, g, \star, g+2\}, \text{ or} \\
v(\Lambda)=(0,1,1,0,0,1) \text{ and  } Gaps(\Lambda)&=&\{1,\dots, g-3, \star,g-1, g, \star,\star, g+3\}, \text{ or} \\
v(\Lambda)=(0,1,0,1,0,1) \text{ and  } Gaps(\Lambda)&=&\{1,\dots, g-3, \star,g-1,\star, g+1,\star, g+3\}.\\
\end{eqnarray*}
\end{enumerate}
In this description, the $\star$ means that the corresponding element is missing in the sequence.
\end{prop}

\begin{proof} Set $\ell=g-t$. The case (a) when $\ell=0$ is covered by Proposition \ref{gist}. For the rest,
by Proposition \ref{prop:type}, it suffices to find the vectors in $v\in \{0, 1\}^{2\ell}$ having $\ell$ entries equal to $0$ and $cotype(v)=\ell$, for $\ell=1,2,3$.

When $\ell=1$, the only possible gap vectors are $(0,1)$ and $(1,0)$. Since $cotype(0,1)=1$ and $cotype(1,0)=0$ then $(1,0)$ is the only stable gap vector.

When $\ell=2$, there are $6$ possible gap vectors: $$(1,1,0,0), (1,0,1,0), (1,0,0,1), (0,1,1,0), (0,1,0,1), (0,0,1,1).$$ Their respective cotypes are equal $0, 1, 2, 2, 2$ and $3$. Therefore, the three stable  gap vectors are $(1,0,0,1), (0,1,1,0), (0,1,0,1)$.

When $\ell=3$, it is routine to check that among the $20$ possible $(0/1)$-vectors of length $6$ with $3$ zero entries, only the  $7$ vectors listed at part $(d)$, have  cotype equal to $3$.
\end{proof}

In the next proposition we describe for each   integer $\ell >0$ several families of stable gap vectors in $\{0,1\}^{2\ell}$.  

\begin{prop} \label{propo:2} Let $\ell\ge 1$ be an integer and let $v=(v_1,\dots, v_{2\ell})$  be a vector with $\ell$ entries equal to $0$ and $\ell$ entries equal to $1$. 
\begin{enumerate}
\item[(i)] Assume $v_i=0$ for a unique $1\leq i\leq \ell$. Then, $v$ is a stable gap vector if and only if 
\begin{enumerate}
\item $i=1$, or
\item $2\leq i\leq \ell$ and $v_{2\ell}=1$.
\end{enumerate}
\item[(ii)] Assume $v_1=\dots = v_{\ell-2}=1$ and that $v_{\ell+m}=v_{\ell+n}=1$ for $0\leq m <n \leq \ell$. Then, $v$ is a stable gap vector if and only if
\begin{enumerate}
\item $n=\ell-1$ and $\ell/2-1\leq m\leq \ell-2$, or
\item $n=\ell$ and $0\leq m \leq \ell/2-1.$
\end{enumerate}
\item[(iii)] Assume $v_2=v_4=\dots =v_{2\ell}=1$. Then, $v$ is a stable gap vector.
\item[(iv)] Assume $\ell \geq 5,  v_2=\dots=v_{\ell-2}=v_\ell=1$, and that $v_{\ell+m}=v_{\ell+n}=1$ for some $1\leq m<n\leq \ell$. Then, $v$ is a stable gap vector if and only if 
\begin{enumerate}
\item $1\leq m <n\leq \ell-2$ with $m\neq \ell-3$ and $n\neq \ell-3$, or
\item $m=1$ and $n=\ell-1$, or
\item $\ell \geq 6$, $m=2$ and $n=\ell$.
\end{enumerate} 
 
\item[(v)] Assume  $v_2=\dots=v_{\ell -1}=1$, and that $v_{\ell+m}=v_{\ell+n}=1$ for some $1\leq m<n\leq \ell$. Then, $v$ is a stable gap vector if and only if 
\begin{enumerate}
\item $2\leq m<n\leq \ell -2$, or 
\item $m=1$ and $n\neq \ell-1$.
\end{enumerate}
\end{enumerate} 
\end{prop}

\begin{proof}
For each of the formats analyzed, we have to chararacterize when $cotype(v)=\ell$.
We denote $\mathcal{D}=\{ j-i:1\leq i<j\leq 2\ell \text{ with } v_i=0 \text{ and } v_j=1\}$, so $cotype(v)=|\mathcal{D}|$.

$(i)$] It follows that $v_k=1$ for all $1\leq k\leq \ell$ with $k\neq i$. Since there are $\ell$ entries equal to $1$ in $v$, we have $v_{\ell+m}=1$ for a unique $1\leq m\leq \ell$, and $v_{\ell+n}=0$ for all $1\leq n\leq \ell$ where $n\neq m$. So  $\mathcal{D}$ contains $\ell+m-i$,  the numbers $1,\dots, \ell-i$ if $i<\ell$, and $1,\dots, m-1$ if $m\geq 2$. Hence, $|\mathcal{D}|\leq 1+\max\{\ell -i, m-1\}\leq \ell$. Therefore, if $v$ is a stable gap vector then $\ell-i=\ell-1$, or $m-1=\ell-1$.

If $i=1$ then $\mathcal{D}=\{1,\dots, \ell-1, \ell-1+m\}$, and $v$ is a stable gap vector for all $1\leq m\leq \ell$.

If $i>1$ and $m=\ell$, then $\mathcal{D}=\{1,\dots, \ell -1, 2\ell-i\}$, and $v$ is a stable gap vector.

$(ii)$] From the hypothesis $v_{\ell -1}=0$. So, $\mathcal{D}=\{1,\dots, m+1\}\cup (\{1,\dots, n+1\}\setminus\{n-m\})$.  

If $n-m \leq m+1$, that is, $n\leq 2m+1$, then $|\mathcal{D}|=n+1$. Hence $v$ is a stable gap vector if and only if $n=\ell-1$ and $\ell/2-1\leq m <n=\ell-1$.

If $n-m> m+1$, that is, $n\geq 2m+2$, then $|\mathcal{D}|=n$. Hence $v$ is a stable gap vector if and only if $n=\ell$ and $0\leq m\leq (n-2)/2$.

$(iii)$] This follows from Example \ref{ex:cotype} $(iii)$.

$(iv)$] Clearly, $ v_1=v_{\ell-1}= 0$. Considering the differences $j-i$ in $\mathcal{D}$ where $i<j \leq \ell$ we obtain that $\{1,\dots, \ell-3, \ell-1\}\subseteq\mathcal{D}$.  When $j\in \{ \ell+m, \ell+n\}$ and $i\in \{1, \ell-1\}$ we get that also $m+1, n+1, \ell+m-1, \ell+n-1 \in \mathcal{D}$.
We already notice $\ell$ distinct values in $\mathcal{D}$. So, $v$ is a stable gap vector if and only if 
\begin{equation}\label{somed}
\mathcal{D}=\{1,\dots, \ell-3, \ell-1, \ell+m-1, \ell+n-1\}.
\end{equation}

If $n\leq   \ell-2$ then for any difference $j-i$ in $\mathcal{D}$  with $\ell<i<j\leq 2\ell$ and $v_j=1, v_i=0$ we have  $j-i\leq \ell-3$. So, in this case we have that $cotype(v)=\ell$ if and only if $m\neq \ell-3$ and $n\neq \ell-3$.

Assume $n=\ell-1$. If $cotype(v)=\ell$, then $v_{\ell+1}=1$ and $m=1$, otherwise $\ell-2=(2\ell-1)-(\ell+1) \in \mathcal{D}$, a contradiction with \eqref{somed}.  
    Conversely, if $m=1$ it is immediate to check that \eqref{somed} holds. 

Assume $n=\ell$. If $cotype(v)=\ell$ then $v_{\ell+2}=1$ and $m=2$, arguing as before.  Note that if $\ell=5$ then  $m+1=3 = \ell-2$, so \eqref{somed} does not hold.  However, it is straightforward to check that once $\ell >5$ and $m=2$ then \eqref{somed} is verified.

$(v)$] Clearly, $v_1=v_\ell=0$. As in the proof of $(v)$, considering the differences $j-i$  where $i\leq \ell$ and $v_i=0$, $v_j=1$,  we obtain that $\{1,2, \dots, \ell-2,  m, n, \ell+m-1, \ell+n-1\} \subseteq \mathcal{D}$.
So, $v$ is a stable gap vector if and only if
\begin{equation} \label{somed2}
\mathcal{D}=\{1,2,\dots, \ell-2, \ell+m-1, \ell+n-1\}.
\end{equation}
 
Assume $m>1$. Then $\ell<\ell+m-1<\ell+n-1$. So, if \eqref{somed2} holds then $m<n\leq \ell -2$. Conversely, it is easy to verify that \eqref{somed2} holds when $n\leq \ell-2$.

If $m=1$ and \eqref{somed2} holds, then $n\neq \ell-1$. Conversely, it is easy to check that whenever $2\leq n \leq \ell$ and $n\neq \ell-1$, equation  \ref{somed2} holds.
\end{proof}

\begin{thm}\label{thm:main2}
Let $\ell \geq 6$ and  $g\geq 3\ell -1$ be integers. Then, $n(g, g-\ell)\geq \ell^2-3\ell+10 $.
\end{thm}

\begin{proof}
We count the stable gap vectors presented in Proposition \ref{propo:2}. 
There are $\ell$ vectors for $(i)(a)$,  $\ell-1$ vectors for $(i)(b)$ and $1$ vector for $(iii)$.

For $(ii)$, we note that for each $0\leq m\leq \ell -2$ and $m\neq \ell/2$ we may find a unique $n$ in $\{ \ell-1, \ell \}$ such that $v$ is a stable gap vector. If $\ell$ is even, then for $m=\ell/2$, $n$ can be any of $\ell$ or $\ell-1$ in a stable gap vector. Therefore, for $(ii)$ we identified $\ell-1$, or $\ell$ vectors depending on $\ell$ being odd, or even, respectively.

For $(iv)$, ${\ell -3\choose 2} +2$ stable gap vectors are described, while for $(v)$ $ {\ell-3 \choose 2} +\ell-2$ more vectors are given.

The only overlap between the cases in Proposition \ref{propo:2} is the vector which occurs at $(i) (b)$ for $i=\ell-1$ and at the same time at $(ii) (b)$ for $m=0, n=\ell$.
Hence, if $\ell \geq 6$ and $g\geq 3\ell -1$ then
$$\begin{array}{ll}
n(g, g-\ell)& \geq 2\ell-1 +(\ell-1)+1+\left({\ell -3\choose 2} +2\right)+\left({\ell-3 \choose 2} +\ell-2\right)-1\\
&=\ell^2-3\ell+10.
\end{array}$$
\end{proof}

\section{{Unimodality} and leaves}\label{sec:uni-leave}
A sequence $a_0,\dots ,a_n$ of real numbers is said to be {\em unimodal} if for some $0\le j\le n$ we have $a_0\le a_1\le \cdots \le a_j\ge a_{j+1}\ge \cdots \ge a_n$.

\begin{thm} The sequence $n(g,1), n(g,2),\dots ,n(g,g)$ is unimodal for each $1\le g\le 33$.
\end{thm}

\begin{proof} {This} result can be checked from values in Table \ref{table1}.
\end{proof}

We notice that  for fixed $g$ the  maximum of $n(g,t)$ is attained when $t$ is around $\lceil \frac{g}{4}\rceil$; see circled values in Table \ref{table1}. 

\begin{conj} Let $1\le t\le g$ be integers. Then, the sequence $n(g,1), n(g,2),\dots ,n(g,g)$ is unimodal. 
\end{conj}

Let $n(g)$ be the number of numerical semigroups of genus $g$. There is a great deal of literature to understand the asymptotic behavior of the sequence $n(g)$. Bras-Amor\'os \cite{BA} conjectured that 
$$\hbox{for each $g\ge 1, \ n(g-2)+n(g-1)\le n(g)$.}$$

In fact, the following far weaker version of the conjecture is still open
$$\hbox{for each $g\ge 1, \ n(g)< n(g+1)$.}$$

On this direction, we put forward the following type-version

\begin{conj} For each $2\le t\le g$, $n(g,t)<n(g+1,t)$.
\end{conj}

Table \ref{table1} supports this conjecture for each $2\le t\le 33$. Notice that the sequence $n(g,1)$ (corresponding to symmetric semigroups) is not increasing, for instance, $n(22,1)>n(23,1)$.
\smallskip

Let $v_i(\Lambda)$ be the number of descendants of $\Lambda$ of type $i$. One may wonder whether the sequence $v_i(\Lambda)$ is increasing (or unimodal). The latter is not true in general. Indeed, consider $\Lambda=\langle 7,9,10,12,13,15\rangle$ with $t(\Lambda)=5, g(\Lambda)=8$ and $F(\Lambda)=11$. The descendants of $\Lambda$ are :
$$\hbox{$\Lambda'=\Lambda\setminus\{12\}= \langle 7,9,10,13,15 \rangle$ with $t(\Lambda')=4$,}$$
$$\hbox{$\Lambda''=\Lambda\setminus\{13\}= \langle 7,9,10,12,15 \rangle$ with $t(\Lambda'')=4$,}$$ 
$$\hbox{$\Lambda'''=\Lambda\setminus\{15\}= \langle 7,9,10,12,13 \rangle$ with $t(\Lambda''')=2$.}$$ 

We thus have that $v_1(\Lambda)=0,v_2(\Lambda)=1,v_3(\Lambda)=0$ and $v_4(\Lambda)=2$.
\smallskip

Moreover, if $\mathcal{F}_{g,t}$ denotes the family of numerical semigroups of type $t$ and genus $g$ then 
the sequence $w_i(\mathcal{F}_{g,t})=\sum\limits_{\Lambda\in\mathcal{F}_{g,t}} v_i(\Lambda)$ is not necessarily either increasing or unimodal, for instance, $w_1(\mathcal{F}_{8,3})=1, w_2(\mathcal{F}_{8,3})=1, w_3(\mathcal{F}_{8,3})=3, w_4(\mathcal{F}_{8,3})=2, w_5(\mathcal{F}_{8,3})=4,
w_6(\mathcal{F}_{8,3})=5, w_7(\mathcal{F}_{8,3})=0, w_8(\mathcal{F}_{8,3})=0$ and $w_9(\mathcal{F}_{8,3})=0$.

\subsection{Leaves}\label{subsec:leaves}

In  \cite{BA-B,BB}  the infinite chains in $\mathcal T$ were investigated. For the latter, it is thus natural to study the nature of the semigroups that are leaves of $\mathcal T$.
The semigroups $\langle 2,2n+1\rangle, n\ge 1$ are symmetric (or equivalently of type 1). They are called {\em hyperelliptic} semigroups. In \cite[Lemma 3]{BA-B} was proved that non-hyperelliptic symmetric semigroups are leaves. Nevertheless, there are non-symmetric semigroups that are also leaves. {The first non-symmetric leaves in $\mathcal T$ are the genus 6 semigroups $\Lambda_1=\langle 5,6,8\rangle, \Lambda_2=\langle5,6,7\rangle$ and $\Lambda_3=\langle4,7,9\rangle$ of type 2 (having $\{7,9\}, \{8,9\}$ and $\{5,10\}$ as pseudo-Frobenius numbers respectively). The other 5 leaves of genus 6 are type 1 (symmetric).}
\smallskip

On this direction, we study the type of semigroups that are leaves. Let $\Lambda$ be a numerical semigroup of genus $g$. If $\Lambda$ is a leaf, what can we say about $t(\Lambda)$? 
Our calculations outcome the values in Table \ref{table:3}

\begin{table}[htbp]
\centering
{\scriptsize
\begin{tabular}{c|c|c|c|c|c|c|c|c|c|c|c|c|c|c|c|c|c|c|c|c|c|c|c|c}
$g$  & 7 &  8 &  9 &  10 &  11 &  12 &13 & 14 &  15 & 16 &  17 &  18 &  19 &  20 &  21 &  22 &  23 &  24 &  25 &  26 &  27 &  28 &  29 & 30 \\  
\hline
$t\le$ & $3$ &  $3$ &  $3$ &  $4$ &  $5$ &  $5$ &  $5$ &  $6$ &  $7$ &  $7$ & $7$ &  $8$ &  $9$ &  $9$ &  $9$ &  $10$ &  $11$ &  $11$ &  $11$ &  $12$ &  $13$ &  $13$ &  $13$ & $14$\\
\end{tabular}
}
\caption{Bounds on the type $t$ of semigroups of genus $g$ that are leaves.}
\label{table:3}
\end{table} 

These values support the following

\begin{conj} Let $g\ge 2$ be an integer such that $g=4q+r, \ 0\le r<4, \ q\ge 0$. If a numerical semigroup $\Lambda$ of genus $g$ is a leaf then
$$t(\Lambda)\le \left\{
\begin{array}{ll}
2q-1 & \text{ if } r=0,1,\\
2q & \text{ if } r=2,\\
2q+1 & \text{ if } r=3.\\
\end{array}\right.$$
\end{conj}



Let $\ell(g,t)$ be the number of semigroups of genus $g$ and type $t$ that are leaves.

\begin{conj} Let $1\le t\le g$ be integers. Then, the sequence $\ell(g,1), \ell(g,2),\dots ,\ell(g,g)$ is unimodal.
\end{conj}

Table \ref{table2} supports this conjecture for each $1\le g\le 33$.
\smallskip

Let $\ell(g)$ be the number of semigroups of genus $g$ that are leaves. It is intriguing that the ratio $\frac{\ell(g)}{n(g)}$ increases very slowly as $g$ grows, see Table \ref{table3}.

\begin{prob} Investigate the asymptotic behavior of $\frac{\ell(g)}{n(g)}$. More precisely, is it true that $\lim\limits_{g\rightarrow \infty} \frac{\ell(g)}{n(g)}$  {exists}? More ambitious, is it strictly less than $\frac{1}{2}$ ?
\end{prob}

{A similar problem, concerning numerical semigroups $\Lambda$ with $F(\Lambda)\le 3m(\Lambda)$, is mentioned in
\cite[Conjecture 4.1]{Z}; see also \cite[Conjecture 6]{Ka1}.
\medskip

It might be interesting either to use other algorithms in the literature (for instance the depth-first search, possibly parallel, algorithm given in \cite{FH}) or, {alternatively, searching specific} implementations  of the existing semigroup packages to push much further the above computations.}

{\scriptsize
\begin{table}[htbp]
\centering{
\begin{tabular}{|c|c|c|c|}
\hline
$g$ & $\ell(g)$ & $n(g)$ & $\ell(g)/n(g)$ \\
\hline
1 & 0 & 1 & 0 \\
2 & 0 & 2 & 0 \\
3 & 1 & 4 & 0,25 \\
4 & 2 & 7 & 0,2857 \\
5 & 2 & 12 & 0,1667 \\
6 & 8 & 23 & 0,3478 \\
7 & 12 & 39 & 0,3077 \\
8 & 20 & 67 & 0,2985 \\
9 & 37 & 118 & 0,3136 \\
10 & 72 & 204 & 0,3529 \\
11 & 110 & 343 & 0,3207 \\
12 & 211 & 592 & 0,3564 \\
13 & 350 & 1001 & 0,3497 \\
14 & 590 & 1693 & 0,3485 \\
15 & 1021 & 2857 & 0,3574 \\
16 & 1742 & 4806 & 0,3625 \\
17 & 2901 & 8045 & 0,3606 \\
18 & 4927 & 13467 & 0,3659 \\
19 & 8244 & 22464 & 0,3670 \\
20 & 13751 & 37396 & 0,3677 \\
21 & 22959 & 62194 & 0,3692 \\
22 & 38356 & 103246 & 0,3715 \\
23 & 63560 & 170963 & 0,3718 \\
24 & 105479 & 282828 & 0,3729 \\
25 & 174673 & 467224 & 0,3739 \\
26 & 288699 & 770832 & 0,3745 \\
27 & 476507 & 1270267 & 0,3751 \\
28 & 786075 & 2091030 & 0,3759 \\
29 & 1293751 & 3437839 & 0,3763 \\
30 & 2127773 & 5646773 & 0,3768 \\
31 & 3495791 & 9266788 & 0,3772 \\
32 & 5738253 & 15195070 & 0,3776 \\
33 & 9409859 & 24896206 & 0,3780 \\
\hline
\end{tabular}
}
\caption{Values of $\ell(g)$, $n(g)$ and $\ell(g)/n(g)$ for each $1\le g\le 33$.}\label{table3}
\end{table} 
}


\begin{table}[htbp]
\centering
\rotatebox{90}{
\resizebox{\textheight}{!}{
\begin{tabular}{c|c|c|c|c|c|c|c|c|c|c|c|c|c|c|c|c|c|c|c|c|c|c|c|c|c|c|c|c|c|c|c|c|c|c|}
\diagbox[width=5em]{genus}{type} & 1 &  2 & 3& 4& 5& 6& 7& 8&9 &10 &11 &12 &13 &14 &15 &16 &17 &18 &19 &20 &21 &22 &23 &24 &25 &26 &27 &28 &29 &30 &31 &32  &33\\ 
\toprule
1& \circled{1} & & & & & & & & & & & & & & & & & & & & & & & & & & & & & & & & \\ 
2& \circled{1}& {\bf 1} & & & & & & & & & & & & & & & & & & & & & & & & & & & & & & & \\																	
3& \circled{2} &{\bf 1} & 1& & & & & & & & & & & & & & & & & & & & & & & & & & & & & & \\																
4&\circled{3}& 2 &{\bf 1} & 1& & & & & & & & & & & & & & & & & & & & & & & & & & & & &\\						
5&3 & \circled{4}& {\bf 3} &{\bf 1}&1& & & & & & & & & & & & & & & & & & & & & & & & & & & & \\
6&\circled{6} & 6	 & 6	&{\bf 3} &{\bf 1} &1& & & & & & & & & & & & & & & & & & & & & & & & & & &\\																
7&8 & 9	 &\circled{12}	&5&	{\bf 3} &{\bf 1}	&1& & & & & & & & & & & & & & & & & & & & & & & & & & \\													
8&7& 17& \circled{20}&11&{\bf 7}&{\bf 3} &{\bf 1}&1& & & & & & & & & & & & & & & & & & & & & & & & &\\																	
9&15&21&\circled{35}&21&14&{\bf 7}&{\bf 3} &{\bf 1}&1& & & & & & & & & & & & & & & & & & & & & & & &\\																	
10&20&32&\circled{56}&40&32&12&{\bf 7}&{\bf 3} &{\bf 1}&1& & & & & & & & & & & & & & & & & & & & & & &\\																
11&18&47&\circled{94}&69&60&28&{\bf 15}&{\bf 7}&{\bf 3} &{\bf 1}&1& & & & & & & & & & & & & & & & & & & & & &\\																
12&36&63&\circled{136}&132&108&57&33&{\bf 15}&{\bf 7}&{\bf 3} &{\bf 1}&1& & & & & & & & & & & & & & & & & & & & &\\
13&44&88&\circled{217}&210&202&111&70&32&{\bf 15}&{\bf 7}&{\bf 3} &{\bf 1}&1& & & & & & & & & & & & & & & & & & & &\\															
14&45&136&	322&	343&\circled{361}&219&	137&	68&{\bf 35}&{\bf 15}&{\bf 7}&{\bf 3} &{\bf 1}&1& & & & & & & & & & & & & & & & & & &\\									
15&83&159&	470&	565&	\circled{619}&405&	289&	130&	75&{\bf 35}&{\bf 15}&{\bf 7}&{\bf 3} &{\bf 1}&1& & & & & & & & & & & & & & & & & &\\							
16&109&227&672&	897&	\circled{1039}&754&533	&289	&152&72&{\bf 35}&{\bf 15}&{\bf 7}&{\bf 3} &{\bf 1}&1& & & & & & & & & & & & & & & & &\\								
17&101&329&	1025&1309&\circled{1730}&1364&	1026&546&328	&147	&{\bf 78}&{\bf 35}&{\bf 15}&{\bf 7}&{\bf 3} &{\bf 1}&1& & & & & & & & & & & & & & & &\\								
18&174&413&1376&2090&\circled{2778}&2392&	1879	&1103&651&312&159&{\bf 78}&{\bf 35}&{\bf 15}&{\bf 7}&{\bf 3} &{\bf 1}&1& & & & & & & & & & & & & & &\\									
19&246&562&	1938	&3067&\circled{4441}&4087&3426	&2122&1298&626&358&153&{\bf 78}&{\bf 35}&{\bf 15}&{\bf 7}&{\bf 3} &{\bf 1}&1& & & & & & & & & & & & & &\\							
20&227&	812&	2810&4422&\circled{6967}&6909&	6024&4043&2523&1307&	709&	342&	{\bf 161}&{\bf 78}&{\bf 35}&{\bf 15}&{\bf 7}&{\bf 3} &{\bf 1}&1& & & & & & & & & & & & &\\				21&420&948&3798&6706&10613&\circled{11318}&10684&7399&4871&2617&1474&684&361&{\bf 161}&{\bf 78}&{\bf 35}&{\bf 15}&{\bf 7}&{\bf 3} &{\bf 1}&1& & & & & & & & & & & &\\			
22&546&1331&5179&9644&16265&\circled{18321}&18134&13642&9226&5168&2956&1440&733&360&{\bf 161}&{\bf 78}&{\bf 35}&{\bf 15}&{\bf 7}&{\bf 3} &{\bf 1}&1& & & & & & & & & & &\\										
23&498&	1827&7463&13360&	24522&29297&	\circled{30564}&24245&	17344&10052&	5927&2922&1546&728&{\bf 367}&{\bf 161}&{\bf 78}&{\bf 35}&{\bf 15}&{\bf 7}&{\bf 3} &{\bf 1}&1& & & & & & & & & &\\								
24&926&	2200&9672&19844&	36037&45754&	\circled{50845}&42662&	31726&19539&	11522&5966&3198&	1519&750&{\bf 367}&{\bf 161}&{\bf 78}&{\bf 35}&{\bf 15}&{\bf 7}&{\bf 3} &{\bf 1}&1& & & & & & & & &\\		
25&1182&3031&13210&	27564&53279&	70929&\circled{82681}&	73752&57493&	36975&22589&	11895&6476&3151&	1607&742&{\bf 367}&{\bf 161}&{\bf 78}&{\bf 35}&{\bf 15}&{\bf 7}&{\bf 3} &{\bf 1}&1& & & & & & & &\\					
26&1121&4207&18579&37735&77889&109057&\circled{132168}&125194&102947&69067&43209&23899&12970&6457&3323&1585&{\bf 757}&{\bf 367}&{\bf 161}&{\bf 78}&{\bf 35}&{\bf 15}&{\bf 7}&{\bf 3} &{\bf 1}&1& & & & & & &\\				

27&2015	&4874&24100&	54527&111716&163904&\circled{211078}&208858&180140&	127562&82539&46611&26054&13089&6892&3258&1625&{\bf 757}&{\bf 367}&{\bf 161}&{\bf 78}&{\bf 35}&{\bf 15}&{\bf 7}&{\bf 3} &{\bf 1}&1&&&&&&\\

28&2496&6774&32223&	74629&161352&245919&	328811&\circled{344606}&311816&231314&155044&90898&	51330&26707&	13900&6767&3407&	1612&{\bf 757}&{\bf 367}&{\bf 161}&{\bf 78}&{\bf 35}&{\bf 15}&{\bf 7}&{\bf 3} &{\bf 1}&1&&&&&\\	

29&2436&9096&44777&	99780&230203&366597&	508326&\circled{558563}&531771&413415&288413&174711&100730&53306&28428&13785&7078&3367&{\bf 1632}&{\bf 757}&{\bf 367}&{\bf 161}&{\bf 78}&{\bf 35}&{\bf 15}&{\bf 7}&{\bf 3} &{\bf 1}&1&&&&\\	

30&4350	&10965&56778&142264&	322409&536007&780577&\circled{896660}&890706&730168&528017&	332219&196498&105970&56741&28499&14478&6985&3425&{\bf 1632}&{\bf 757}&{\bf 367}&{\bf 161}&{\bf 78}&{\bf 35}&{\bf 15}&{\bf 7}&{\bf 3} &{\bf 1}&1&&&\\

31&5602&14520&76367&191216&457167&779733&1181200&1419227&\circled{1476568}&1266870&958160&623766&377946&209567&113636&57428&29860&14305&7175&3418&{\bf 1632}&{\bf 757}&{\bf 367}&{\bf 161}&{\bf 78}&{\bf 35}&{\bf 15}&{\bf 7}&{\bf 3} &{\bf 1}&1&&\\

32&5317	&20329&105381&252790&639033&1139142&1763982&2221672&\circled{2415743}&2169709&1709518&1161149&722149&409209&225282&116301&60364&29614&14745&7148&{\bf 3436}&{\bf 1632}&{\bf 757}&{\bf 367}&{\bf 161}&{\bf 78}&{\bf 35}&{\bf 15}&{\bf 7}&{\bf 3} &{\bf 1}&1&\\	

33&8925&23282&133674&355362&882069&1625034&2645417&3435192&\circled{3894557}&3668138&3016976&2127885&1367186&794338&443863&233104&121816&60317&30733&14631&7214&{\bf 3436}&{\bf 1632}&{\bf 757}&{\bf 367}&{\bf 161}&{\bf 78}&{\bf 35}&{\bf 15}&{\bf 7}&{\bf 3} &{\bf 1}&1\\
\end{tabular}
}}
\caption{Values of $n(g,t)$ for each $1\le t\le g\le 33$. {The} largest value in each row is circled, and stability numbers are bolded.}\label{table1}
\end{table} 


\begin{table}[htbp]
\centering
\rotatebox{90}{
\resizebox{\textheight}{!}{
\begin{tabular}{c|c|c|c|c|c|c|c|c|c|c|c|c|c|c|c|c|c|c|c|c|c|c|c|c|c|c|c|c|c|c|c|c|c|c|}
\diagbox[width=5em]{genus}{type} & 1 &  2 & 3& 4& 5& 6& 7& 8&9 &10 &11 &12 &13 &14 &15 &16 &17 &18 &19 &20 &21 &22 &23 &24 &25 &26 &27 &28 &29 &30 &31 &32  &33\\ 
\hline
1 & 0 &  &  &  &  &  &  &  &  &  &  &  &  &  &  &  &  &  &  &  &  &  &  &  &  &  &  &  &  &  &  &  &  \\
2 & 0 & 0 &  &  &  &  &  &  &  &  &  &  &  &  &  &  &  &  &  &  &  &  &  &  &  &  &  &  &  &  &  &  &  \\
3 & 1 & 0 & 0 &  &  &  &  &  &  &  &  &  &  &  &  &  &  &  &  &  &  &  &  &  &  &  &  &  &  &  &  &  &  \\
4 & 2 & 0 & 0 & 0 &  &  &  &  &  &  &  &  &  &  &  &  &  &  &  &  &  &  &  &  &  &  &  &  &  &  &  &  &  \\
5 & 2 & 0 & 0 & 0 & 0 &  &  &  &  &  &  &  &  &  &  &  &  &  &  &  &  &  &  &  &  &  &  &  &  &  &  &  &  \\
6 & 5 & 3 & 0 & 0 & 0 & 0 &  &  &  &  &  &  &  &  &  &  &  &  &  &  &  &  &  &  &  &  &  &  &  &  &  &  &  \\
7 & 7 & 4 & 1 & 0 & 0 & 0 & 0 &  &  &  &  &  &  &  &  &  &  &  &  &  &  &  &  &  &  &  &  &  &  &  &  &  &  \\
8 & 6 & 11 & 3 & 0 & 0 & 0 & 0 & 0 &  &  &  &  &  &  &  &  &  &  &  &  &  &  &  &  &  &  &  &  &  &  &  &  &  \\
9 & 14 & 14 & 9 & 0 & 0 & 0 & 0 & 0 & 0 &  &  &  &  &  &  &  &  &  &  &  &  &  &  &  &  &  &  &  &  &  &  &  &  \\
10 & 19 & 25 & 26 & 2 & 0 & 0 & 0 & 0 & 0 & 0 &  &  &  &  &  &  &  &  &  &  &  &  &  &  &  &  &  &  &  &  &  &  &  \\
11 & 17 & 37 & 47 & 8 & 1 & 0 & 0 & 0 & 0 & 0 & 0 &  &  &  &  &  &  &  &  &  &  &  &  &  &  &  &  &  &  &  &  &  &  \\
12 & 35 & 56 & 83 & 33 & 4 & 0 & 0 & 0 & 0 & 0 & 0 & 0 &  &  &  &  &  &  &  &  &  &  &  &  &  &  &  &  &  &  &  &  &  \\
13 & 43 & 76 & 148 & 70 & 13 & 0 & 0 & 0 & 0 & 0 & 0 & 0 & 0 &  &  &  &  &  &  &  &  &  &  &  &  &  &  &  &  &  &  &  &  \\
14 & 44 & 123 & 228 & 143 & 51 & 1 & 0 & 0 & 0 & 0 & 0 & 0 & 0 & 0 &  &  &  &  &  &  &  &  &  &  &  &  &  &  &  &  &  &  &  \\
15 & 82 & 146 & 362 & 283 & 141 & 6 & 1 & 0 & 0 & 0 & 0 & 0 & 0 & 0 & 0 &  &  &  &  &  &  &  &  &  &  &  &  &  &  &  &  &  &  \\
16 & 108 & 216 & 540 & 518 & 311 & 45 & 4 & 0 & 0 & 0 & 0 & 0 & 0 & 0 & 0 & 0 &  &  &  &  &  &  &  &  &  &  &  &  &  &  &  &  &  \\
17 & 100 & 312 & 853 & 818 & 656 & 147 & 15 & 0 & 0 & 0 & 0 & 0 & 0 & 0 & 0 & 0 & 0 &  &  &  &  &  &  &  &  &  &  &  &  &  &  &  &  \\
18 & 173 & 400 & 1194 & 1426 & 1242 & 423 & 68 & 1 & 0 & 0 & 0 & 0 & 0 & 0 & 0 & 0 & 0 & 0 &  &  &  &  &  &  &  &  &  &  &  &  &  &  &  \\
19 & 245 & 541 & 1714 & 2256 & 2216 & 1033 & 230 & 8 & 1 & 0 & 0 & 0 & 0 & 0 & 0 & 0 & 0 & 0 & 0 &  &  &  &  &  &  &  &  &  &  &  &  &  &  \\
20 & 226 & 794 & 2535 & 3366 & 3936 & 2159 & 682 & 49 & 4 & 0 & 0 & 0 & 0 & 0 & 0 & 0 & 0 & 0 & 0 & 0 &  &  &  &  &  &  &  &  &  &  &  &  &  \\
21 & 419 & 926 & 3484 & 5339 & 6559 & 4195 & 1813 & 212 & 12 & 0 & 0 & 0 & 0 & 0 & 0 & 0 & 0 & 0 & 0 & 0 & 0 &  &  &  &  &  &  &  &  &  &  &  &  \\
22 & 545 & 1313 & 4838 & 7982 & 10661 & 8041 & 4101 & 789 & 84 & 2 & 0 & 0 & 0 & 0 & 0 & 0 & 0 & 0 & 0 & 0 & 0 & 0 &  &  &  &  &  &  &  &  &  &  &  \\
23 & 497 & 1796 & 7036 & 11291 & 17266 & 14311 & 8642 & 2418 & 294 & 7 & 2 & 0 & 0 & 0 & 0 & 0 & 0 & 0 & 0 & 0 & 0 & 0 & 0 &  &  &  &  &  &  &  &  &  &  \\
24 & 925 & 2179 & 9207 & 17264 & 26553 & 24839 & 17124 & 6252 & 1069 & 62 & 5 & 0 & 0 & 0 & 0 & 0 & 0 & 0 & 0 & 0 & 0 & 0 & 0 & 0 &  &  &  &  &  &  &  &  &  \\
25 & 1181 & 3006 & 12681 & 24497 & 40881 & 41798 & 32492 & 14280 & 3594 & 238 & 25 & 0 & 0 & 0 & 0 & 0 & 0 & 0 & 0 & 0 & 0 & 0 & 0 & 0 & 0 &  &  &  &  &  &  &  &  \\
26 & 1120 & 4178 & 17932 & 34006 & 62198 & 68824 & 58382 & 30926 & 9929 & 1094 & 107 & 3 & 0 & 0 & 0 & 0 & 0 & 0 & 0 & 0 & 0 & 0 & 0 & 0 & 0 & 0 &  &  &  &  &  &  &  \\
27 & 2014 & 4842 & 23402 & 49962 & 91855 & 109834 & 103543 & 61905 & 24560 & 4166 & 408 & 14 & 2 & 0 & 0 & 0 & 0 & 0 & 0 & 0 & 0 & 0 & 0 & 0 & 0 & 0 & 0 &  &  &  &  &  &  \\
28 & 2495 & 6747 & 31462 & 69240 & 136027 & 173885 & 176834 & 118288 & 56267 & 13328 & 1412 & 78 & 12 & 0 & 0 & 0 & 0 & 0 & 0 & 0 & 0 & 0 & 0 & 0 & 0 & 0 & 0 & 0 &  &  &  &  &  \\
29 & 2435 & 9055 & 43844 & 93362 & 199051 & 270115 & 294755 & 219955 & 118964 & 36206 & 5644 & 334 & 31 & 0 & 0 & 0 & 0 & 0 & 0 & 0 & 0 & 0 & 0 & 0 & 0 & 0 & 0 & 0 & 0 &  &  &  &  \\
30 & 4349 & 10936 & 55785 & 134633 & 283435 & 410920 & 483129 & 394911 & 239431 & 89875 & 18821 & 1360 & 186 & 2 & 0 & 0 & 0 & 0 & 0 & 0 & 0 & 0 & 0 & 0 & 0 & 0 & 0 & 0 & 0 & 0 &  &  &  \\
31 & 5601 & 14477 & 75255 & 182259 & 408813 & 616662 & 775025 & 685918 & 464957 & 204831 & 55432 & 5946 & 590 & 22 & 3 & 0 & 0 & 0 & 0 & 0 & 0 & 0 & 0 & 0 & 0 & 0 & 0 & 0 & 0 & 0 & 0 &  &  \\
32 & 5316 & 20294 & 104066 & 242269 & 580380 & 927332 & 1212984 & 1170758 & 866410 & 437022 & 146484 & 22736 & 2041 & 149 & 12 & 0 & 0 & 0 & 0 & 0 & 0 & 0 & 0 & 0 & 0 & 0 & 0 & 0 & 0 & 0 & 0 & 0 &  \\
33 & 8924 & 23235 & 132277 & 342815 & 810110 & 1356021 & 1900504 & 1943772 & 1567615 & 888897 & 353474 & 73589 & 8033 & 507 & 85 & 1 & 0 & 0 & 0 & 0 & 0 & 0 & 0 & 0 & 0 & 0 & 0 & 0 & 0 & 0 & 0 & 0 & 0 \\
\end{tabular}
}}
\caption{Values of $\ell(g,t)$ for each $1\le t\le g\le 33$.}\label{table2}
\end{table}

\end{document}